\documentclass[12pt,reqno]{amsart}
\usepackage{amsmath}
\usepackage{amssymb}
\usepackage[left=2.7cm,top=2.7cm,right=2.7cm,bottom=2.7cm]{geometry}
\usepackage[usenames]{color}
\usepackage{enumerate}
\usepackage[normalem]{ulem}

\begin{document}
\newtheorem{theorem}{Theorem}[section]
\newtheorem{lemma}[theorem]{Lemma}
\newtheorem{definition}[theorem]{Definition}
\newtheorem{conjecture}[theorem]{Conjecture}
\newtheorem{proposition}[theorem]{Proposition}
\newtheorem{claim}[theorem]{Claim}
\newtheorem{algorithm}[theorem]{Algorithm}
\newtheorem{corollary}[theorem]{Corollary}
\newtheorem{observation}[theorem]{Observation}
\newtheorem{problem}[theorem]{Open Problem}
\newtheorem{question}[theorem]{Question}
\newcommand{\noin}{\noindent}
\newcommand{\ind}{\indent}
\newcommand{\om}{\omega}
\newcommand{\pp}{\mathcal P}
\newcommand{\AC}{\mathcal A \mathcal C}
\newcommand{\bAC}{\overline{\AC}}
\newcommand{\ppp}{\mathfrak P}
\newcommand{\N}{{\mathbb N}}
\newcommand{\Z}{{\mathbb Z}}
\newcommand{\LL}{\mathbb{L}}
\newcommand{\R}{{\mathbb R}}
\newcommand{\E}{\mathbb E}
\newcommand{\Prob}{\mathbb{P}}
\newcommand{\eps}{\varepsilon}
\newcommand{\ram}[1]{\hat{R}({#1})}
\newcommand{\G}{{\mathcal{G}}}

\newcommand{\Ss}{{\mathcal S}}
\newcommand{\Nn}{{\mathcal N}}

\newcommand{\ceil}[1]{\left \lceil #1 \right \rceil}
\newcommand{\floor}[1]{\left \lfloor #1 \right \rfloor}
\newcommand{\size}[1]{\left \vert #1 \right \vert}
\newcommand{\dist}{\mathrm{dist}}

\title{On some multicolour Ramsey properties of random graphs}

\author{Andrzej Dudek}
\address{Department of Mathematics, Western Michigan University, Kalamazoo, MI, USA}
\thanks{The first author was supported in part by Simons Foundation Grant \#244712 and by the National Security Agency under Grant Number H98230-15-1-0172. The United States Government is authorized to reproduce and distribute reprints notwithstanding any copyright notation hereon.}
\email{\tt andrzej.dudek@wmich.edu}

\author{Pawe\l{} Pra\l{}at}
\address{Department of Mathematics, Ryerson University, Toronto, ON, Canada}
\thanks{The second author was supported by NSERC and Ryerson University}
\email{\tt pralat@ryerson.ca}

\begin{abstract}
The size-Ramsey number $\ram{F}$ of a graph $F$ is the smallest integer $m$ such that there exists a graph $G$ on $m$ edges with the property that any colouring of the edges of $G$ with two colours yields a monochromatic copy of $F$. In this paper, first we focus on the size-Ramsey number of a path $P_n$ on $n$ vertices. In particular, we show that $5n/2-15/2 \le \ram{P_n} \le 74n$ for $n$ sufficiently large. (The upper bound uses expansion properties of random $d$-regular graphs.) This improves the previous lower bound, $\ram{P_n} \ge (1+\sqrt{2})n-O(1)$, due to Bollob\'as, and the upper bound, $\ram{P_n} \le 91n$, due to Letzter. Next we study long monochromatic paths in edge-coloured random graph $\G(n,p)$ with $pn \to \infty$. Let $\alpha > 0$ be an arbitrarily small constant. Recently, Letzter showed that a.a.s.\ any $2$-edge colouring of $\G(n,p)$ yields a monochromatic path of length $(2/3-\alpha)n$, which is optimal. Extending this result, we show that a.a.s.\ any $3$-edge colouring of $\G(n,p)$ yields a monochromatic path of length $(1/2-\alpha)n$, which is also optimal. In general, we prove that for $r\ge 4$ a.a.s.\ any $r$-edge colouring of $\G(n,p)$ yields a monochromatic path of length $(1/r-\alpha)n$. We also consider a related problem and show that for any $r \ge 2$, a.a.s.\ any $r$-edge colouring of $\G(n,p)$ yields a monochromatic connected subgraph on $(1/(r-1)-\alpha)n$ vertices, which is also tight.
\end{abstract}

\maketitle

\section{Introduction}
Following standard notations, we write $G\to (F)_r$ if any $r$-edge colouring of $G$ (that is, any colouring of the edges of $G$ with $r$ colours) yields a monochromatic copy of $F$. For simplicity, we often write $G\to F$ instead of $G\to (F)_2$. Furthermore, we define the \emph{size-Ramsey number} of $F$ as $\ram{F,r} = \min \{ |E(G)| : G \to (F)_r \}$ and again, for simplicity, $\ram{F} = \ram{F,2}$. 

We consider the size-Ramsey number of the path $P_n$ on $n$ vertices. It is obvious that $\ram{P_n} = \Omega(n)$ and that $\ram{P_n} = O(n^2)$ (for example, $K_{2n}\to P_n$), but the exact behaviour of $\ram{P_n}$ was not known for a long time. In fact, Erd\H{o}s~\cite{E81} offered \$100 for a proof or disproof that 
\[
\ram{P_n} / n \to \infty \quad \text{ and } \quad \ram{P_n} / n^2 \to 0.
\]
This problem was solved by Beck~\cite{B83} in 1983 who, quite surprisingly, showed that $\ram{P_n} < 900 n$. (Each time we refer to inequality such as this one, we mean that the inequality holds for sufficiently large $n$.) A variant of his proof, provided by Bollob\'{a}s~\cite{B01}, gives $\ram{P_n} < 720 n$. Very recently, the authors of this paper~\cite{DP15} used a different and more elementary argument that shows that $\ram{P_n} < 137 n$. The argument was subsequently tuned by Letzter~\cite{L15} who showed that $\ram{P_n} < 91 n$. On the other hand, the first nontrivial lower bound was provided by Beck~\cite{B90} and his result was subsequently improved by Bollob\'as~\cite{B86} who showed that $\ram{P_n} \ge (1+\sqrt{2}) n - O(1)$.

\smallskip

In Section~\ref{sec2}, we show that for any $r \in \N$, $\ram{P_n,r} \ge \frac{(r+3)r}{4}n - O(r^2)$ (Theorem~\ref{thm:gen_lb}), which slightly improves the lower bound of Bollob\'as~\cite{B86} for two colours and generalizes it to more colours. It follows that $\ram{P_n} \ge 5n/2 - O(1)$. In Section~\ref{sec3}, using expansion properties of random $d$-regular graphs, we show that $\ram{P_n} \le 74n$ (Theorem~\ref{thm:two_wholes_d-reg}) which improves the leading constant provided by Letzter~\cite{L15}. We also generalize our upper bound to more colours, showing that $\ram{P_n,r} \le 33 r 4^r n$ (Theorem~\ref{thm:more_colours_upper}). 

\smallskip

In Section~\ref{sec4}, we deal with the following, closely related problem. It is known, due to Gerencs{\'e}r and Gy{\'a}rf{\'a}s~\cite{GG67}, that $K_n \to P_{(2/3+o(1))n}$; due to Gy{\'a}rf{\'a}s, Ruszink{\'o}, S{\'a}rk{\"o}zy, and Szemer{\'e}di~\cite{GRSS07,GRSS07b} and also Figaj and \L{}uczak~\cite{FL07}, we know that $K_n \to (P_{(1/2+o(1))n})_3$. Moreover, these results are best possible.  Unfortunately, very little is known about the behaviour for more colours; although it is conjectured that $K_n \to (P_{(1/(r-1)+o(1))n})_r$ for $r \in \N \setminus \{1, 2\}$, which would be best possible. Clearly, if for some subgraph $G$ of $K_n$, $G \to P_{cn}$, then $K_n \to P_{cn}$ as well. On the other hand, one could expect that sparse subgraphs of $K_n$ ``arrow'' much shorter paths. However, this intuition seems to be false. As a matter of fact, for two colours Letzter~\cite{L15} showed that a.a.s.\ $\G(n,p) \to P_{(2/3-\alpha)n}$, provided that $pn \to \infty$, which is optimal. (Here and later on, $\alpha > 0$ is an arbitrarily small constant.) We adjust her approach (using also some ideas of Figaj and \L{}uczak~\cite{FL07}) and show that a.a.s.\ $\G(n,p) \to (P_{(1/2-\alpha)n})_3$, provided that $pn \to \infty$, which is also optimal. For any $r \in \N \setminus \{1, 2,3\}$ we prove that a.a.s.\ $\G(n,p) \to (P_{(1/r-\alpha)n})_r$, provided that $pn \to \infty$ (Theorem~\ref{thm:multi_Gnp}). This is, perhaps, not sharp but it is a consequence of the poor current understanding of the behaviour for $K_n$. On the other hand, note that the best one can hope for is that a.a.s.\ $\G(n,p) \to (P_{(1/(r-1)+o(1))n})_r$, provided that $pn \to \infty$, since there are $r$-colourings of the edges of $K_n$ (and so also of $\G(n,p)$) with no monochromatic path of length $n/(r-1)$.

\smallskip

In the next section, Section~\ref{sec5}, we continue with similar direction but relax the property of having $P_{cn}$ as a subgraph to having a component of size $cn$. It is known, due to Gy{\'a}rf{\'a}s~\cite{G77} and F\"uredi~\cite{F81}, that for any $r$-colouring of the edges of $K_n$, there is a monochromatic component of order $(1/(r-1)+o(1))n$. Moreover, this is best possible if $r-1$ is a prime power. We show that $K_n$ and $\G(n,p)$ behave very similarly with respect to the size of the largest monochromatic component. More precisely, we prove that a.a.s.\ for any $r$-colouring of the edges of $\G(n,p)$, there is a monochromatic component of order $(1/(r-1)-\alpha)n$, provided that $pn \to \infty$ (Theorem~\ref{thm:mono_comps}). As before, this result is clearly best possible. 

\section{Lower bound on the size-Ramsey number of $P_n$}\label{sec2}

In this section, we improve the lower bound (for two colours) given by Bollob\'as~\cite{B86} who showed that $\ram{P_n} \ge (1+\sqrt{2})(n-1) - 4$. In our result, the leading constant $(1+\sqrt{2})$ is increased to $5/2$. Moreover, we provide a more general result that holds for any number of colours $r$, which improves the trivial lower bound $\ram{P_n,r} \ge (r-1)(n-1)+1$. 

\begin{theorem}\label{thm:gen_lb}
Let $r\ge 1$. Then, for all sufficiently large $n$
\[
\ram{P_n,r} \ge \frac{(r+3)r}{4}n-\frac{r(5r+11)}{4}+3.
\] 
\end{theorem}

We will need the following auxiliary claim.

\begin{claim}\label{claim:tree}
Let $k\in \N \cup \{0\}$ and $T$ be a tree. Then, at least one of the following two properties holds:
\begin{enumerate}[(i)]
\item\label{claim:i} $T$ has $k$ edges $e_1, e_2, \dots,e_k$ such that $T-\{e_1, e_2, \dots,e_k\}$ contains no $P_n$,
\item\label{claim:ii} $T$ contains $(k+2)$ vertex-disjoint connected subgraphs of order at least $\lfloor n/2\rfloor$ each.
\end{enumerate}
\end{claim}
\begin{proof}
We prove the statement by induction on $k$. For $k=0$, if \eqref{claim:i} fails, then $T$ contains a copy of $P_n$ and we are done. Indeed, after splitting the path as equally as possible we get two components of the desired order so \eqref{claim:ii} holds. 

Let $k\in \N \cup \{0\}$ and suppose that the statement holds for any integer $i$ satisfying $0\le i\le k$. Again, assume that \eqref{claim:i} fails for $(k+1)$; that is, for any choice of $e_1, e_2, \dots,e_{k+1}$, $T-\{e_1, e_2, \dots,e_{k+1}\}$ contains $P_n$. We will show that \eqref{claim:ii} must hold; that is, $T$ contains $(k+3)$ vertex-disjoint connected subgraphs of order at least $\lfloor n/2\rfloor$ each.

Clearly $T\supseteq P_n$. Hence, let $e$ be such that $T-e$ consists of two components, $T_1$ and $T_2$, each of order at least $\lfloor n/2\rfloor$. By assumption we made (that \eqref{claim:i} fails for $(k+1)$), for any choice of $k_1$ edges $e_1, e_2, \dots,e_{k_1}$ in $T_1$ and $k_2$ edges $f_1, f_2, \dots,f_{k_2}$ in $T_2$ such that $k_1+k_2=k$, either $T_1-\{e_1, e_2, \dots,e_{k_1}\}$ or $T_2-\{f_1, f_2, \dots,f_{k_2}\}$ contains $P_n$. 

If $T_1-\{e_1,e_2,\dots,e_{k_1}\}\supseteq P_n$ and $T_2-\{f_1,f_2,\dots,f_{k_2}\}\supseteq P_n$ for any choice of the edges, then (by inductive hypothesis) $T_1$ and $T_2$ have, respectively, $(k_1+2)$ and $(k_2+2)$ vertex-disjoint connected subgraphs of size $\lfloor n/2\rfloor$, giving $k_1+k_2+4\ge k+3$ vertex-disjoint connected subgraphs of order $\lfloor n/2\rfloor$ in $T$. Therefore, without loss of generality, we may assume that $T_2-\{f_1,f_2,\dots,f_{k_2}\} \nsupseteq P_n$ for some choice of $f_1,f_2,\dots,f_{k_2}$, where $k_2$ is as small as possible. Of course, this implies that $T_1-\{e_1, e_2, \dots,e_{k_1}\} \supseteq P_n$ for any choice of the edges. Now, we need to consider two cases.
If $k_2=0$, then (by inductive hypothesis) $T_1$ has $(k_1+2)$ vertex-disjoint connected subgraphs of order $\lfloor n/2\rfloor$ which, together with $T_2$ yield $(k+3)$ desired large subgraphs in~$T$. On the other hand, if $k_2\ge 1$, then (due to minimality of $k_2$) we infer that for any choice of $f_1,f_2,\dots,f_{k_2-1}$, $T_2-\{f_1,f_2, \dots,f_{k_2-1}\}\supseteq P_n$. Thus, (again, by inductive hypothesis) $T$ has $(k_1+2) + (k_2-1+2) = k+3$ vertex-disjoint connected subgraphs of order $\lfloor n/2\rfloor$, as needed. 
\end{proof}

Now, we are ready to prove the main result of this section. The proof is based on ideas from the proof from~\cite{B86} and \cite{B90}. 

\begin{proof}[Proof of Theorem~\ref{thm:gen_lb}]
We prove the statement by induction on $r$. For $r=1$ the desired inequality is trivially true: $\ram{P_n,1} \ge n-1$. Assume that the statement holds for some $r\in \N$ and, for a contradiction, suppose that it fails for $(r+1)$, that is,
\[
\ram{P_n,r+1} < \frac{(r+4)(r+1)}{4}n-\frac{(r+1)(5(r+1)+11)}{4}+3.
\]
Let $G=(V,E)$ be a graph of order $N$ and size $\ram{P_n,r+1}$, such that $G\to(P_n)_{r+1}$. Clearly, $G$ is connected. We will independently deal with two cases, depending on $N$. 

\smallskip
\emph{Case 1}: $N > (r+2)(n-3)/2$. Let $T$ be any spanning tree of $G$. We apply Claim~\ref{claim:tree} with $k=r$. First, let us assume that property (i) in the claim holds; that is, $T$ has $r$ edges $e_1, e_2, \dots,e_{r}$ such that $T-\{e_1, e_2, \dots,e_{r}\}$ contains no $P_n$. We colour all $(N-1)-r$ edges in $T-\{e_1,e_2,\dots,e_{r}\}$ using the first colour. The number of uncoloured edges is at most 
\begin{align*}
\ram{P_n,r+1} - &(N-r-1) \\
&<   \frac{(r+4)(r+1)}{4}n-\frac{(r+1)(5(r+1)+11)}{4}+3 - \frac{r+2}{2}(n-3)+r+1\\
&= \frac{(r+3)r}{4}n-\frac{r(5r+11)}{4}+3 \le \ram{P_n,r},
\end{align*}
where the last inequality follows from the inductive hypothesis. Thus, we can colour the uncoloured edges with the remaining $r$ colours in such a way that there is no monochromatic $P_n$. Consequently, $G\not \rightarrow(P_n)_{r+1}$, which gives us the desired contradiction.

Assume then that property (ii) in the claim holds; that is, $T$ contains $(r+2)$ vertex-disjoint connected subgraphs of order at least $\lfloor n/2\rfloor$ each. We colour $\lfloor n/2\rfloor-1$ edges of each of the $(r+2)$ components with the first colour. (If some component has more than $\lfloor n/2\rfloor-1$ edges, we select edges to colour arbitrarily.) The number of uncoloured edges is at most
\begin{align*}
\ram{P_n,r+1} - &(r+2)\left(\left\lfloor\frac{n}{2}\right\rfloor-1\right)  \\
&<   \frac{(r+4)(r+1)}{4}n-\frac{(r+1)(5(r+1)+11)}{4}+3 - (r+2)\left(\frac{n}{2}-\frac{3}{2}\right)  \\
&= \frac{(r+3)r}{4}n-\frac{r(5r+11)}{4}+3-(r+1) < \ram{P_n,r},
\end{align*}
and this yields a contradiction ($G\not \rightarrow(P_n)_{r+1}$), as before. 

\smallskip
\emph{Case 2}: $N \le (r+2)(n-3)/2$. Let $U\subseteq V$ be any set of size $|U| = n-1$, and let $W_1, W_2, \dots, W_r$ be an equipartition of $V\setminus U$. Clearly, for any $1\le i\le r$,
\[
|W_i| \le \left\lceil \frac{1}{r}\left( \frac{r+2}{2}(n-3) - (n-1)  \right)  \right\rceil
= \left\lceil  \frac{n-3}{2} - \frac{2}{r} \right\rceil < \frac{n-1}{2} - \frac{2}{r}.
\]
Let $G_i$ be a bipartite subgraph of $G$ induced by the edges between $W_i$ and $W_{i+1}\cup\dots\cup W_r\cup U$. We colour the edges of $G_i$ with the $i$-th colour and the remaining edges (inside $U$ or $W_i$'s) with the last colour. Clearly there is no monochromatic (or, in fact, any) copy of $P_n$ in $U$ or $W_i$'s. Furthermore, each path in $G_i$ must alternate between $W_i$ and $W_{i+1}\cup\dots\cup W_r\cup U$. Thus, the longest path in $G_i$ has at most $2|W_i|+1 <  n$ vertices. We get the desired contradiction ($G\not \rightarrow(P_n)_{r+1}$) for the last time and the proof is finished.
\end{proof}

\section{Upper bound on the size-Ramsey number of $P_n$}\label{sec3}

In this section, we present various upper bounds on $\ram{P_n}$. Corresponding theorems use different approaches and different probability spaces. Subsection~\ref{subsec:Letzter} uses the existing lemma of Letzter. In Subsection~\ref{subsec:two_wholes}, another improvement is developed, which gives the strongest bound. Finally, Subsection~\ref{subsec:multi} deals with more colours.

\smallskip

Let us recall a few classic models of random graphs that we study in this section and later on in the paper. The \emph{binomial random graph} $\G(n,p)$ is the random graph $G$ with vertex set $[n] := \{ 1, 2, \ldots, n\}$ in which every pair $\{i,j\} \in \binom{[n]}{2}$ appears independently as an edge in $G$ with probability~$p$. The \emph{binomial random bipartite graph} $\G(n,n, p)$ is the random bipartite graph $G=(V_1 \cup V_2, E)$ with partite sets $V_1, V_2$, each of order $n$, in which every pair $\{i,j\} \in V_1 \times V_2$ appears independently as an edge in $G$ with probability~$p$. Note that $p=p(n)$ may (and usually does) tend to zero as $n$ tends to infinity.  

Recall that an event in a probability space holds \emph{asymptotically almost surely} (or \emph{a.a.s.}) if the probability that it holds tends to $1$ as $n$ goes to infinity. Since we aim for results that hold a.a.s., we will always assume that $n$ is large enough. For simplicity, we do not round numbers that are supposed to be integers either up or down; this is justified since these rounding errors are negligible to the asymptomatic calculations we will make. Finally, we use $\log n$ to denote natural logarithms.

\smallskip

However, our main results in this section refer to another probability space, the probability space of random $d$-regular graphs with uniform probability distribution. This space is denoted $\mathcal{G}_{n,d}$, and asymptotics are for $n\to\infty$ with $d\ge 2$ fixed, and $n$ even if $d$ is odd.

Instead of working directly in the uniform probability space of random regular graphs on $n$ vertices $\mathcal{G}_{n,d}$, we use the \textit{pairing model} (also known as the \textit{configuration model}) of random regular graphs, first introduced by Bollob\'{a}s~\cite{bollobas2}, which is described next. Suppose that $dn$ is even, as in the case of random regular graphs, and consider $dn$ points partitioned into $n$ labelled buckets $v_1,v_2,\ldots,v_n$ of $d$ points each. A \textit{pairing} of these points is a perfect matching into $dn/2$ pairs. Given a pairing $P$, we may construct a multigraph $G(P)$, with loops allowed, as follows: the vertices are the buckets $v_1,v_2,\ldots, v_n$, and a pair $\{x,y\}$ in $P$ corresponds to an edge $v_iv_j$ in $G(P)$ if $x$ and $y$ are contained in the buckets $v_i$ and $v_j$, respectively. It is an easy fact that the probability of a random pairing corresponding to a given simple graph $G$ is independent of the graph, hence the restriction of the probability space of random pairings to simple graphs is precisely $\mathcal{G}_{n,d}$. Moreover, it is well known that a random pairing generates a simple graph with probability asymptotic to $e^{-(d^2-1)/4}$ depending on $d$, so that any event holding a.a.s.\ over the probability space of random pairings also holds a.a.s.\ over the corresponding space $\mathcal{G}_{n,d}$. For this reason, asymptotic results over random pairings suffice for our purposes. For more information on this model, see, for example, the survey of Wormald~\cite{NW-survey}.

\smallskip

Also, we will be using the following well-known concentration inequality. Let $X \in \textrm{Bin}(n,p)$ be a random variable with the binomial distribution with parameters $n$ and $p$. Then, a consequence of Chernoff's bound (see, for example,~\cite[Corollary~2.3]{JLR}) is that 
$$
\Prob( |X-\E X| \ge \eps \E X) \le 2\exp \left( - \frac {\eps^2 \E X}{3} \right)  
$$
for  $0 < \eps < 3/2$. 

\subsection{Existing approach}\label{subsec:Letzter}

Using the following (deterministic) lemma Letzter showed that $\ram{P_n}< 91 n$.

\begin{lemma}[\cite{L15}]\label{lem:letzter}
Let $G$ be a graph of order $cn$ for some $c>2$. Assume that for every two disjoint sets of vertices $S$ and $T$ such that $|S| = |T| = n(c-2)/4$ we have $e(S,T) \neq 0$. Then, $G\to P_n$.
\end{lemma}
In fact, she showed that a.a.s.\ $\G(cn, d/n) \to P_n$ with $c=4.86$ and $d=7.7$. This is an improved version of a result of the authors of this paper~\cite{DP15} and a very similar result of Pokrovskiy~\cite{P14}. Here we show that a slightly stronger bound can be obtained if random $d$-regular graphs are used.

\begin{theorem}\label{thm:one_whole_d-reg}
Let $c = 5.219$ and $d = 30$. Then, a.a.s.\ $\G_{cn,d} \to P_n$, which implies that $\ram{P_n} < 78.3 n$ for sufficiently large $n$.
\end{theorem}
\begin{proof}
Consider $\mathcal{G}_{cn,d}$ for some $c \in (2, \infty)$ and $d \in \N$. Our goal is to show that (for a suitable choice of $c$ and $d$) the expected number of pairs of two disjoint sets, $S$ and $T$, such that $|S|=|T|=n(c-2)/4$ and $e(S,T) = 0$ tends to zero as $n \to \infty$. This, together with the first moment principle, implies that a.a.s.\ no such pair exists and so, by Lemma~\ref{lem:letzter}, we get that a.a.s.\ $\mathcal{G}_{cn,d} \to P_n$. As a result, $\ram{P_n} \le (cd/2 +o(1)) n$. 

Let $a=a(n)$ be any function of $n$ such that $adn \in \Z$ and $0\le a\le (c-2)/4$, and let $X(a)$ be the expected number of pairs of two disjoint sets $S, T$ such that $|S|=|T|=n(c-2)/4$, $e(S,T) = 0$, and $e(S,V \setminus (S \cup T)) = adn$. Using the paring model, it is clear that 
\begin{eqnarray*}
X(a) &=& {cn \choose \frac {c-2}{4} n} { cn - \frac {c-2}{4} n  \choose \frac {c-2}{4} n } { \frac {c-2}{4} dn \choose adn } { \frac {c+2}{2} dn \choose adn} M \left( \frac {c-2}{4} dn - adn \right) (adn)! \\
&& \quad \cdot \ M \left( \frac {c+2}{2} dn - adn + \frac {c-2}{4} dn \right) / M(cdn),
\end{eqnarray*}
where $M(i)$ is the number of perfect matchings on $i$ vertices, that is, 
$$
M(i) = \frac {i!} {(i/2)! 2^{i/2}}.
$$
(Each time we deal with perfect matchings, $i$ is assumed to be an even number.) After simplification we get
\begin{eqnarray*}
X(a) &=& (cn)! \left( \frac {c-2}{4} dn \right)! \left( \frac {c+2}{2} dn \right)! \left( \frac {3c+2}{4} dn - adn \right)! 2^{cdn/2} (cdn/2)! \\
&& \quad \cdot \ \Bigg[ \left( \frac{c-2}{4} n \right)!^2 \left( \frac{c+2}{2} n \right)! \ 2^{( \frac{c-2}{4} dn - adn)/2} \left( \left( \frac{c-2}{4} dn - adn\right) /2 \right)! (adn)! \\
&& \quad \quad \quad \left( \frac {c+2}{2} dn - adn \right)! \ 2^{( \frac{3c+2}{4} dn - adn)/2} \left( \left( \frac{3c+2}{4} dn - adn \right)/2 \right)! (cdn)! \Bigg]^{-1}.
\end{eqnarray*}
Using Stirling's formula ($i! \sim \sqrt{2\pi i} (i/e)^i$) and focusing on the exponential part we obtain
$$
X(a) = \Theta( n^{-3/2} ) e^{f(a,c,d)n},
$$
where
\begin{eqnarray*}
f(a,c,d) &=& c \left(1- \frac d2 \right) \log c + \frac {c-2}{4} (d-2) \log \left( \frac {c-2}{4} \right) + \frac {c+2}{2} (d-1) \log \left( \frac {c+2}{2} \right) \\
&& \quad - \left( \frac {c-2}{4} - a \right) \frac {d}{2} \log \left( \frac {c-2}{4} - a \right) - ad \log a \\
&& \quad - \left( \frac {c+2}{2} - a \right) d \log \left( \frac {c+2}{2} - a \right) + \left( \frac {3c+2}{4} - a \right) \frac {d}{2} \log \left( \frac {3c+2}{4} - a \right).
\end{eqnarray*}
Thus, if $f(a,c,d) \le 0$ for any integer $adn$ under consideration, then $X(a) = O(n^{-3/2}) = o(n^{-1})$. We would get $\sum_{adn} X(a) = o(1)$ (as $adn = O(n)$), the desired property would be satisfied, and the proof would be finished.

It is straightforward to see that 
$$
\frac{\partial f}{\partial a} = - \frac {d}{2} \Big( 2 \log 2 - \log (c-2-4a)+2 \log a - 2 \log(c+2-2a)+\log(3c+2-4a) \Big).
$$
Now, since $\frac{\partial f}{\partial a} = 0$ if and only if $a^2 - ca + (c^2-4)/8 = 0$, function $f(a,c,d)$ has a local maximum for $a = a_0 := c/2 - \sqrt{2c^2 + 8}/4$, which is also a global one on $a\in(-\infty, c/2 +\sqrt{2c^2 + 8}/4)$. Since $a\le (c-2)/4 < c/2 +\sqrt{2c^2 + 8}/4$, we get that
$$
f(a,c,d) \le g(c,d) := f(a_0, c, d).
$$ 
Finally, by taking $c = 5.219$ and $d = 30$, we get $g(c,d) < -0.0005$ and the proof of the first part is finished. Finally, it follows that $\ram{P_n} < 78.3 n$ for $n$ large enough, as $cd/2 = 78.285 < 78.3$. (Of course, constants $c$ and $d$ were chosen as to minimize $cd/2$, provided that $g(c,d) \le 0$.)
\end{proof}

Lemma~\ref{lem:letzter} provides a sufficient condition for $G \to P_n$ that is quite convenient for any good expander $G$. On the other hand, it is not so difficult to see that it can never give an upper bound better than $26.4n$. Indeed, let $\alpha = (c-2)/(4c)$ and $G$ be a graph of order $N=cn$ and average degree $d$ such that for every two disjoint sets of vertices $S$ and $T$ with $|S| = |T| = \alpha N$ we have $e(S,T) \neq 0$. Then the complement of $G$ contains no copy of $K_{\alpha N, \alpha N}$ and the well-known K\H{o}v\'ari, S\'os and Tur\'an~\cite{KST54} inequality (see also Theorem~11 in~\cite{B98}) yields
\[
N \binom{N-1-d}{\alpha N} \le (\alpha N-1) \binom{N}{\alpha N},
\]
which for $N$ sufficiently large implies that $d \ge \frac{\log \alpha}{\log (1-\alpha)} -1$. Thus, the number of edges in $G$ is at least 
\[
\frac{Nd}{2} = \frac{cnd}{2} \ge \frac{c}{2} \left( \frac{\log \alpha}{\log (1-\alpha)} -1\right)n = f(c)n,
\]
where
\[
f(c) := \frac{c}{2}\left( \frac{\log (c-2)/(4c)}{\log (3c+2)(4c) -1}\right).
\]
The above function takes a minimum at $c= c_0 \approx 5.633$ which gives $f(c_0) \approx 26.415$. 

\subsection{Improved approach} \label{subsec:two_wholes}

In this subsection, we provide another sufficient condition for $G\to P_n$ which can be viewed as a slight straightening of Lemma~\ref{lem:letzter}. We start with the following elementary observation that is similar to the one in~\cite{DP15} and~\cite{P14}.

\begin{lemma}\label{lem:obs}
Let $G$ be a graph of order $cn$ for some $c>1$. Then, the vertex set $V(G)$ can be partitioned into three sets $P, U,W$, $|U|=|W|=(cn-|P|)/2$ such that the graph induced by $P$ has a Hamiltonian path and $e(U, W) = 0$.
\end{lemma}
\begin{proof}
We perform the following algorithm on $G$ and construct a path $P$. Let $v_1$ be an arbitrary vertex of $G$, let $P=(v_1)$, $U = V(G) \setminus \{v_1\}$, and $W = \emptyset$. If there is an edge from $v_1$ to $U$ (say from $v_1$ to $v_2$), we extend the path as $P=(v_1,v_2)$ and remove $v_2$ from $U$. We continue extending the path $P$ this way for as long as possible. It might happen that we reach the point of the process in which $P$ cannot be extended, that is, there is a path from $v_1$ to $v_k$ (for some $k \le cn$) and there is no edge from $v_k$ to $U$. If this is the case, $v_k$ is moved to $W$ and we try to continue extending the path from $v_{k-1}$, perhaps reaching another critical point in which another vertex will be moved to $W$, etc. If $P$ is reduced to a single vertex $v_1$ and no edge to $U$ is found, we move $v_1$ to $W$ and simply re-start the process from another vertex from $U$, again arbitrarily chosen. 

An obvious but important observation is that during this algorithm there is never an edge between $U$ and $W$. Moreover, in each step of the process, the size of $U$ decreases by 1 or the size of $W$ increases by 1. Hence, at some point of the process both $U$ and $W$ must have equal size, namely, $|U|=|W|=(cn-|P|)/2$. We stop the process and $P, U, W$ form the desired partition of $V(G)$.
\end{proof}

Now we are ready to state the main tool used in this subsection.

\begin{lemma}\label{lem:2_colours}
Let $G$ be a graph of order $cn$ for some $c>2$. Assume that for every four disjoint sets of vertices $S_1,S_2,T_1,T_2$ such that $|S_1|+|S_2| = |T_1|+|T_2| =|S_1|+|T_1| = |S_2|+|T_2| = n(c-2)/2$ we have $e(S_1,T_2) \neq 0$ or $e(S_2,T_1) \neq 0$. Then, $G\to P_n$.
(Clearly, this implies that $|S_1|=|T_2|$ and $|S_2|=|T_1|$.)
\end{lemma}
\begin{proof}
Suppose that $G \not\to P_n$; that is, suppose that it is possible to colour the edges of $G$ with the colours blue and red such that there is no monochromatic~$P_n$. Let $G_b$ be the graph on the vertex set $V(G)$, induced by blue edges. It follows from Lemma~\ref{lem:obs} (applied to $G_b$) that there exist two disjoint sets $U,W \subseteq V(G_b) = V(G)$ each of size $n(c-1)/2$ such that there is no blue edge between $U$ and $W$ (observe that $|P| < n$ as there is no blue $P_n$ in $G$). Now, consider a bipartite graph $G_r=(U \cup W, E_r)$, with partite sets $U,W$, and $E_r  = \{ uw \in E(G) : u \in U, w \in W \}$.
Clearly, all edges of $G_r$ are red. Lemma~\ref{lem:obs} (this time applied to $G_r$) implies then that there exist two disjoint sets $U',W' \subseteq V(G_r) \subseteq V(G)$ each of size $n(c-2)/2$ such that there is no red edge between $U'$ and $W'$ (again, observe that $|P'| < n$ as there is no red $P_n$ in $G \supseteq G_r$). Moreover, as $G_r$ is bipartite, the path $P'$ has at most $n/2$ vertices in $U$ and at most $n/2$ vertices in $W$. Hence, we may assume that $|(U' \cup W') \cap U| = |(U' \cup W') \cap W| = n(c-2)/2$. Let $S_1 = U \cap U'$, $S_2 = U \cap W'$, $T_1 = W \cap U'$, and $T_2 = W \cap W'$. Clearly, $|S_1|+|S_2| = |T_1|+|T_2| =|S_1|+|T_1| = |S_2|+|T_2| = n(c-2)/2$, $e(S_1,T_2) = 0$, and $e(S_2,T_1) = 0$. The proof of the theorem is finished.
\end{proof}

First, we will check how the new lemma performs for binomial random graphs.
\begin{theorem}
Let $c = 5.28$ and $d = 6$. Then, a.a.s.\ $\G(cn,d/n) \to P_n$, which implies that $\ram{P_n} < 83.7 n$ for sufficiently large $n$.
\end{theorem}
\begin{proof}
Consider $\G(cn,d/n)$. Let $X$ be the number of (ordered) quadruples of disjoint sets $S_1,S_2,T_1,T_2$ such that $|S_1|+|S_2| = |T_1|+|T_2| =|S_1|+|T_1| = |S_2|+|T_2| = n(c-2)/2$ and $e(S_1,T_2) = e(S_2,T_1)= 0$.
Then,
\[
\E(X) = \binom{cn}{\frac{c-2}{2}n} \binom{cn-\frac{c-2}{2}n}{\frac{c-2}{2}n}\sum_{s=0}^{\frac{c-2}{2}n} \binom{\frac{c-2}{2}n}{s}\binom{\frac{c-2}{2}n}{\frac{c-2}{2}n-s} \left( 1 - \frac{d}{n}\right)^{s^2 + \left(\frac{c-2}{2}n-s\right)^2}.
\]
Since $s^2 + (\frac{c-2}{2}n-s)^2 \ge 2\left(\frac{c-2}{4}n\right)^2$ and $\sum_{s=0}^{m} \binom{m}{s}\binom{m}{m-s} = \binom{2m}{m}$, we get
\begin{align*}
\E(X) &\le \binom{cn}{\frac{c-2}{2}n} \binom{cn-\frac{c-2}{2}n}{\frac{c-2}{2}n} \binom{(c-2)n}{\frac{c-2}{2}n}  \left( 1 - \frac{d}{n}\right)^{2\left(\frac{c-2}{4}n\right)^2}\\
&= \frac{(cn)!\cdot ((c-2)n)!}{ \left(\left(\frac{c-2}{2}n\right)!\right)^4 \cdot (2n)!} \left( 1 - \frac{d}{n}\right)^{2\left(\frac{c-2}{4}n\right)^2}\\
&\le \left( \frac{c^c (c-2)^{c-2}}{4\left(\frac{c-2}{2}\right)^{2(c-2)}} \right)^n e^{-\frac{d(c-2)^2}{8}n} = e^{f(c,d)n},
\end{align*}
where
\[
f(c,d):= c\log c + (c-2)\log(c-2) - 2\log 2 - 2(c-2)\log((c-2)/2) - d(c-2)^2/8.
\]
Observe that for $c=5.28$ and $d=6$, $f(c,d)<0$ and so the first part follows by the first moment principle. Finally, it follows immediately from Chernoff's bound that the number of edges is well concentrated around $c^2 d n / 2$. As $c^2 d/2 < 83.7$, we get that $\ram{P_n} < 83.7 n$ for $n$ large enough.
\end{proof}

As expected, random $d$-regular graphs give slightly better constant.

\begin{theorem}\label{thm:two_wholes_d-reg}
Let $c = 5.4806$ and $d = 27$. Then, a.a.s.\ $\G_{cn,d} \to P_n$, which implies that $\ram{P_n} < 74 n$ for sufficiently large $n$.
\end{theorem}
\begin{proof}
Since the proof technique is exactly the same as the proof of Theorem~\ref{thm:one_whole_d-reg}, we only provide a sketch of the proof here. Consider $\mathcal{G}_{cn,d}$ for some $c \in (2, \infty)$ and $d \in \N$. Let $s=s(n)$, $a=a(n)$, $b=b(n)$, $t=t(n)$ be any integer-valued functions of $n$ such that $0\le s \le (c-2)/4$, $0 \le a \le s$, $0 \le b \le s$, $0 \le t \le \min\{(c-2)/2-a-b, 2\}$. Let $X(s,a,b,t)$ be the expected number of (ordered) quadruples of disjoint sets $S_1,S_2,T_1,T_2$ such that $|S_1|=|T_2|=sn$, $|S_2|=|T_1|=((c-2)/2-s)n$, $e(S_1,T_2) = e(S_2,T_1)= 0$, $e(S_1,T_1) = adn$, $e(S_2,T_2) = bdn$,  and $e(S_1 \cup S_2,V \setminus (S_1 \cup S_2 \cup T_1 \cup T_2)) = tdn$.  (Note that, in particular, $|S_1|+|S_2| = |T_1|+|T_2| =|S_1|+|T_1| = |S_2|+|T_2| = n(c-2)/2$.)

Using the paring model, we get that 
\begin{eqnarray*}
X(s,a,b,t) &=& {cn \choose sn} {(c-s)n \choose (\frac {c-2}{2} -s) n} { \frac {c+2}{2} n  \choose sn } { (\frac {c+2}{2} - s)n \choose (\frac {c-2}{2} -s) n} { sdn \choose adn } { (\frac {c-2}{2} - s) dn \choose adn} (adn)! \\
&& \quad \cdot { sdn \choose bdn } { (\frac {c-2}{2} - s) dn \choose bdn} (bdn)! { ( \frac{c-2}{2}-a-b)dn \choose tdn } { 2dn \choose tdn } (tdn)! \\
&& \quad \cdot M \left( \left(\frac {c-2}{2} - a - b - t \right) dn \right) M \left( \left(\frac {c+2}{2} - a - b - t \right) dn \right) / M(cdn).
\end{eqnarray*}
Our goal is to show that $X(s,a,b,t) = o(n^{-4})$ (regardless of the choice of $s,a,b,t$) so that $\sum_{s,a,b,t} X(s,a,b,t) = o(1)$. Hence, we need to maximize $X(s,a,b,t)$. One can show that the maximum is obtained for $a=b$ and for the case when $|S_1|=|S_2|=|T_1|=|T_2|=s = (c-2)/4$. Therefore, we need to concentrate on 
$$
Y(a,t) = X \left ( \frac {c-2}{4}, a, a, t \right) = e^{f(a,t)n + o(n)},
$$
where
\begin{eqnarray*}
f(a,t) &=& c \log c + 4 (d-1) \left( \frac c4 - \frac 12 \right) \log \left( \frac c4 - \frac 12 \right) + (d-1) 2 \log 2 - 2 d a \log a - d t \log t \\
&& \quad - \frac {d}{2} c \log c - 4 d \left( \frac c4 - \frac 12 - a \right) \log \left( \frac c4 - \frac 12 - a \right) - d (2-t) \log (2-t) \\
&& \quad + d \left( \frac c2 - 1 - 2a \right) \log \left( \frac c2 - 1 - 2a \right) - \frac d2 \left( \frac c2 - 1 - 2a - t \right) \log \left( \frac c2 - 1 - 2a - t\right) \\
&& \quad + \frac d2 \left( \frac c2 + 1 - 2a - t \right) \log \left( \frac c2 + 1 - 2a - t\right).
\end{eqnarray*}
Since $\frac{\partial f}{\partial t} = 0$ if and only if $t^2 - (c-4a)t + (c-2-4a) = 0$, function $f(a,t)$ has a local maximum for $t = t_0 := (c-4a)/2 - \sqrt{ (c-4a)^2 - 4(c-2-4a) }/2$, which is also a global one on the interval under consideration. We get 
$$
f(a,t) \le g(a) := f(a, t_0).
$$ 
Finally, by taking $c = 5.4806$ and $d = 27$, we get $g(a) < -0.0001$ for any $a$ we deal with. It follows that for any choice of parameters,  $X(s,a,b,t) \le Y(a,t) \le \exp(-0.0001n) = o(n^{-4})$, and the proof is finished. It follows that $\ram{P_n} < 74 n$ for $n$ large enough, as $cd/2 = 73.9881 < 74$.
\end{proof}

\subsection{More colours}\label{subsec:multi}

In this subsection, we turn our attention to more than two colours. Here is a natural generalization of Lemma~\ref{lem:2_colours} in easier, bipartite, setting.
\begin{lemma}\label{thm:multi_size_ram_Pn}
Let $r \in \N \setminus \{1\}$ and $G=(V_1\cup V_2, E)$ be a balanced bipartite graph of order $cn$ for some $c > 2^r-1$. Assume that for every two sets $S\subseteq V_1$ and $T\subseteq V_2$, $|S| = |T| = \left( (c + 1)/2^r-1 \right)n/2$, we have $e(S,T) \neq 0$. Then, $G\to(P_n)_r$.
\end{lemma}
\begin{proof}
Suppose that $G \not\to (P_n)_r$; that is, suppose that it is possible to colour the edges of $G$ with the colours from the set $\{1, 2, \ldots, r\}$ such that there is no monochromatic~$P_n$. Let $\beta_i$ be defined recursively as follows: $\beta_0 = c$, $\beta_{i} = (\beta_{i-1}-1)/2$ for $i \ge 1$. Note that $\beta_i = (c+1)/2^i - 1$ for $i \ge 0$. We will use (inductively) Lemma~\ref{lem:obs} to show the following claim, which will finish the proof (by taking $S=S_r$ and $T=T_r$).

\smallskip
\emph{Claim}: For each $i \in \{0, 1, \ldots, r\}$, there exist two sets $S_i \subseteq V_1$ and $T_i \subseteq V_2$, each of size at least $\beta_i n/2$, such that there is no edge between $S_i$ and $T_i$ in colour from the set $\{1, 2, \ldots, i\}$. 

\smallskip
The base case ($i=0$) trivially (and vacuously) holds by taking $S_0=V_1$ and $T_0=V_2$. Suppose that the claim holds for some $i$, $0 \le i < r$. We apply Lemma~\ref{lem:obs} to the bipartite graph with partite sets $S_i, T_i$, induced by the edges in colour $(i+1)$. It follows that $S_i \cup T_i$ can be partitioned into three sets $P, U, W$, $P$ has a Hamiltonian path, $|U|=|W|= (\beta_i n - |P|)/2$, and $e(U,W)=0$. Since $G$ is bipartite, $|S_i \setminus P| = |T_i \setminus P| = (\beta_i n - |P|)/2$. Without loss of generality, we may assume that $| (S_i \setminus P) \cap U) | \ge | (T_i \setminus P) \cap U) |$. As a result, $| (S_i \setminus P) \cap U) | = | (T_i \setminus P) \cap W) | \ge n(\beta_i - |P|)/4 \ge n(\beta_i - 1)/4$. The inductive step is finished by taking $S_{i+1} = (S_i \setminus P) \cap U$ and $T_{i+1} = (T_i \setminus P) \cap W$. 
\end{proof}
 
\begin{theorem}\label{thm:more_colours_upper}
Let $r \in \N \setminus \{1\}$, $c = 2^{r+1}$, and $d=8r$. Then, a.a.s.\ $\G(cn,cn,d/n) \to (P_n)_r$, which implies that $\ram{P_n,r} < 33 r 4^r n$ for sufficiently large $n$.
\end{theorem}
\begin{proof}
Consider $\G(cn,cn,d/n) = (V_1 \cup V_2, E)$. We will show that the expected number of pairs of sets $S\subseteq V_1$ and $T\subseteq V_2$ such that $|S| = |T| = cn/2^{r+2}$ and $e(S,T) = 0$ tends to zero as $n \to \infty$. This will finish the first part of the proof by Lemma~\ref{thm:multi_size_ram_Pn}, combined with the first moment principle, as $cn/2^{r+2} < \left( (c + 1)/2^r-1 \right)n/2$ (recall that $c = 2^{r+1}$). Indeed, the expectation we need to estimate is equal to 
\begin{eqnarray*}
{ cn \choose cn/2^{r+2} }^2 \left( 1 - \frac dn \right)^{ (cn/2^{r+2})^2 } &\le& (2^{r+2} e)^{ 2cn/2^{r+2} } \exp \left( - d \left( \frac {c}{2^{r+2}} \right)^2 n \right) \\
&=& o \left(  (e^{2r})^{ 2cn/2^{r+2} } \exp \left( - d \left( \frac {c}{2^{r+2}} \right)^2 n \right) \right) \\
&=& o \left(  \exp \left( \left( 4r - \frac {dc}{2^{r+2}} \right) \frac {cn}{2^{r+2}} \right) \right) = o(1),
\end{eqnarray*}
as $dc/2^{r+2} = 4r$. The second part follows from the fact that the number of edges in $\G(cn,cn,d/n)$ is well concentrated around $c^2 d n$ and $c^2 d = 32 r 4^r < 33 r 4^r$.
\end{proof}

Summarizing, we showed that there exist some positive constants $c_1, c_2$ such that for any $r \in \N$ we have
\[
c_1 r^2 \cdot  n \le \ram{P_n,r} \le c_2 r 4^r \cdot n.
\]
Of course, one can improve Lemma~\ref{thm:multi_size_ram_Pn} slightly. For example, in the first step there is no need to assume that the graph is bipartite. Also one could try to use the ``double wholes'' approach as in Lemma~\ref{lem:2_colours}. However, the improvement would not be substantial. It would be interesting to determine the order of magnitude of $\ram{P_n,r}$ as a function of $r$ (for fixed $n$). 

\section{Multicoloured path Ramsey number of $\G(n,p)$}\label{sec4}

Determining the classical Ramsey number for paths, $R(P_n,r)$, it is a well-known problem that attracted a lot of attention. The case $r=2$ is well understood, due to the result of Gerencs{\'e}r and Gy{\'a}rf{\'a}s~\cite{GG67}. It is known that
\[
R(P_n,2) = \left\lfloor \frac{3n-2}{2} \right\rfloor.
\]
For $r=3$ and $n$ sufficiently large, Gy{\'a}rf{\'a}s, Ruszink{\'o}, S{\'a}rk{\"o}zy, and Szemer{\'e}di~\cite{GRSS07,GRSS07b} proved that
\[
R(P_n,3) = 
\begin{cases}
2n-1 & \text{ for odd } n,\\
2n-2 & \text{ for even } n,
\end{cases}
\]
as conjectured earlier by Faudree and Schelp~\cite{FS75}. (An asymptotic value was obtained earlier by Figaj and \L{}uczak~\cite{FL07}.) However, this problem is still open for \emph{small} values of $n$. On the other hand, very little is known for any integer $r\ge 4$. The well-known Erd\H{o}s and Gallai result~\cite{EG59} (see Theorem~\ref{thm:erdos_gallai} below) implies only that $R(P_n,r) \le rn$. Very recently, S{\'a}rk{\"o}zy~\cite{S15} improved it and showed that for any integer $r\ge 2$,
\[
R(P_n,r) \le \left( r - \frac{r}{16r^3+1} \right)n.
\] 
It is believed that the value of $R(P_n,r)$ is close to $(r-1)n$.

\smallskip

In this section, we consider an analogous problem for $\G(n,p)$ with average degree, $np$, tending to infinity as $n\to\infty$. We are interested in the following constant:
\begin{equation}\label{eq:cr}
c_r = \sup\{ c \in [0,1] : \G(n,p) \to \left(P_{cn}\right)_r \ \text{a.a.s., provided $np \to \infty$} \}.
\end{equation}
The case $r=2$ is already investigated; due to Letzter~\cite{L15} we know that $c_2=2/3$. For any integer $r\ge 3$, Lemma~\ref{thm:multi_size_ram_Pn} gives only $c_r \ge 1/(2^r-1)$. We will show a stronger result. 

\begin{theorem}\label{thm:multi_Gnp}
Let $r\in \N \setminus \{1, 2, 3\}$, $\alpha > 0$ be an arbitrarily small constant, and $p=p(n)$ be such that $pn\to\infty$. Then, a.a.s.\ $\G(n,p) \to \left(P_{(1/r -\alpha) n}\right)_r$, which implies that $c_r \ge 1/r$.
Furthermore, for 3 colours,  a.a.s.\ $\G(n,p) \to \left(P_{(1/2 -\alpha) n}\right)_3$, which is optimal and implies that $c_3 = 1/2$.
\end{theorem}
\noindent
Furthermore, we conjecture that $c_r = n/R(P_n,r)$ for any $r \ge 2$, which is true for $r=2$~\cite{L15} and for $r=3$, due to the above theorem.

\smallskip

First we prove Theorem~\ref{thm:multi_Gnp} for $r\ge 4$. Let us start with the Erd\H{o}s and Gallai result~\cite{EG59} and its perturbed version.

\begin{theorem}[\cite{EG59}]\label{thm:erdos_gallai}
Let $G$ be a graph of order $n$ with no $P_k$. Then, $|E(G)| \le n(k-2)/2$.
\end{theorem}

After applying this theorem to the subgraph of $G$ induced by the majority colour, we get the following corollary.

\begin{corollary}\label{cor:erdos_gallai}
Let $r\in \N \setminus \{1, 2\}$ and $0<\varepsilon < 1$. Then, for every graph $G$ of order $n$ with at least $(1-\varepsilon)\binom{n}{2}$ edges we have $G\to (P_k)_r$, where $k=(1 - \varepsilon)n/r$.
\end{corollary}

\smallskip

Now we introduce some notation needed to state Sparse Regularity Lemma.
For given two disjoint subsets of vertices $U$ and $W$ in a graph $G$, we define the \emph{$p$-density} of the edges between $U$ and $W$ as 
\[
d_p(U,W) = \frac{e(U,W)}{p|U||W|}.
\]
Moreover, we say that $U,W$ is an \emph{$(\varepsilon,p)$-regular pair} if, for every $U'\subseteq U$ and  $W'\subseteq W$ with $|U'|\ge \varepsilon |U|$, $|W'|\ge \varepsilon |W|$, $|d_p(U',W')-d_p(U,W)|\le \varepsilon$.
Suppose that $0 <\eta< 1$, $D > 1$ and $0 < p<1$ are given. We will say that a graph $G$ is \emph{$(\eta,p,D)$-upper-uniform} if for all disjoint subsets $U_1$ and $U_2$ with $|U_1| \ge |U_2| \ge \eta |V(G)|$, $d_p(U_1,U_2)\le D$. 

\smallskip

The following theorem, which is a variant of Szemer\'edi's Regularity Lemma~\cite{S78} for sparse graphs, was discovered independently by Kohayakawa~\cite{K97} and R\"odl (see, for example,~\cite{C14}).

\begin{theorem}[\emph{Sparse Regularity Lemma}]\label{SRL}
For every $\varepsilon>0$, $r\ge 1$ and $D\ge 1$, there exist $\eta>0$ and $T$ such that for every $0\le p\le 1$, if $G_1, G_2, \dots,G_r$ are $(\eta,p,D)$-upper-uniform graphs on the vertex set $V$, then there is an equipartition of $V$ into $s$ parts, where $1/\varepsilon \le s\le T$, for which all but at most $\varepsilon \binom{s}{2}$ of the pairs induce an $(\eps,p)$-regular pair in each $G_i$. 
\end{theorem}

Now, we are ready to prove the main theorem of this section. Recall that for $r=2$, Letzter~\cite{L15} showed that $c_2=2/3$. The proof below is essentially her approach that easily extends to any numbers of colours.

\begin{proof}[Proof of Theorem~\ref{thm:multi_Gnp} for $r\ge 4$]
Let $r\in \N \setminus \{1, 2, 3\}$, $\alpha > 0$, and $p=p(n)$ be such that $pn\to\infty$ as $n\to\infty$. We will show that a.a.s.\ for every $r$-edge colouring of $G = \G(n,p)=(V,E)$ there is a monochromatic path of length at least $(1/r-\alpha)n$.

Pick $\varepsilon = \eps(\alpha) >0$ such that $(1-8\eps)(1-(r+1) \eps) \ge 1 - r \alpha$ and $1/(2r)>\varepsilon$ and set $D=2$. Apply the sparse regularity lemma with above defined $\varepsilon, D$, and $r$. Let $\eta$ and $T$ be the constants arising from this lemma.

For each $i \in [r]$, let $G_i$ be a subgraph of $G$ induced by the edges coloured with colour~$i$. By Chernoff's bound, for any $U$ and $W$ of size at least $\eta n$, the $p$-density $d_p(U,W)$ in $G$ is at most 2 and so the $p$-density in each $G_i$ is also at most 2. (Indeed, there are obviously at most $(2^n)^2=4^n$ choices for $U$ and $W$, and for each choice the failure probability is at most $2 \exp( - \eta^2 n^2 p/3) = o(4^n)$.) Thus, each $G_i$ is an $(\eta,p,D)$-upper-uniform graph. Consequently, Theorem~\ref{SRL} implies that there is an equipartition of $V=V_1\cup V_2\cup \dots\cup V_s$, where $1/\varepsilon \le s\le T$, for which all but at most $\varepsilon \binom{s}{2}$ of the pairs induce an $(\eps,p)$-regular pair in each $G_i$.

Let $R$ be the auxiliary (cluster) graph with vertex set $[s]$, where $\{i,j\}$ is an edge if and only if $V_i$, $V_j$ induce an $(\eps,p)$-regular bipartite graph in each of the $r$ colours. Colour $\{i,j\}$ in $R$ by the majority colour appearing between $V_i$ and $V_j$ in $G$. Again by Chernoff's bound the $p$-density $d_p(V_i,V_j)$ in $G$ is at least $1/2$. Hence, if $\{i,j\}$ is coloured by $c$, then $d_p(V_i,V_j)$ in $G_c$ is at least $1/(2r)$.

Observe that the number of edges in $R$ is at least $(1-\varepsilon)\binom{s}{2}$. Hence, it follows from Corollary~\ref{cor:erdos_gallai} that $R$ contains a monochromatic, say red, path $P = (i_1, i_2, \dots,i_\ell)$ on at least $\ell = (1-\varepsilon)s/r$ vertices. Furthermore, we divide each set $V_{i_j}$ into two sets $U_j,W_j$ of equal sizes, that is, $|U_j| = |W_j| = n/(2s)$. Let $P_j$ be a longest red path in the bipartite graph $G[U_j, W_{j+1}]$. Since $V_{i_j}$ and $V_{i_{j+1}}$ are $(\varepsilon, p)$-regular with $p$-density at least $1/(2r)$, Lemma~\ref{lem:obs} implies that $P_j$ covers at least $(1-4\varepsilon)n/s$ vertices of $G[U_j,W_{j+1}]$ for each $1\le j\le \ell-1$.

Now, we are going to glue $P_1,P_2,\dots,P_{\ell-1}$, trying to lose as few vertices as possible. Let $X_j$ be the last $\varepsilon n/s$ vertices of $P_{j}$ in $U_j$, and let $Y_{j+1}$ be the first $\varepsilon n/s$ vertices of $P_{j+1}$ in $U_{j+1}$. Since $V_j, V_{j+1}$ is an $(\eps,p)$-regular pair (in the graph induced by red edges) with $p$-density at least $1/(2r)>\varepsilon$, there must be a red edge between $X_j$ and $Y_{j+1}$. Thus, $G$ has a red path $Q$ which contains all vertices of $V(P_1)\cup V(P_2)\cup \dots\cup V(P_{\ell-1})$ but at most $4\varepsilon(\ell-1) n/s$. Consequently,
\begin{align*}
|V(Q)| &\ge (\ell-1)(1-4\varepsilon)n/s - 4\varepsilon (\ell-1)n/s = (1-8\varepsilon) (\ell-1)n/s\\
&\ge (1-8\varepsilon) (1-(r+1) \varepsilon)n/r \ge (1/r - \alpha)n,
\end{align*}
as required.
\end{proof}

Now we show how to prove Theorem~\ref{thm:multi_Gnp} for $r=3$. The proof is based on an ingenious idea of Figaj and \L{}uczak~\cite{FL07} of ``connected matchings'' and relies on the following lemma.

\begin{theorem}[\cite{FL07}]\label{thm:matchings}
Let $0<\varepsilon \le 0.001$ and let $G$ be a graph of order $n$ with at least $(1-\varepsilon)\binom{n}{2}$ edges. Then, for any $3$-colouring of the edges of $G$, there is a monochromatic component which contains a matching saturating at least $(1/2-5\varepsilon^{1/7})n$ vertices.
\end{theorem}

\begin{proof}[Proof of Theorem~\ref{thm:multi_Gnp} for $r=3$]
The proof is very similar to the case $r\ge 4$. Therefore, we only emphasize differences.
As in the previous case we apply the sparse regularity lemma to an $r$-coloured graph $G$ and then Theorem~\ref{thm:matchings} to the $r$-coloured cluster graph~$R$ of order $s$. This way we obtain a monochromatic, say red, minimal component~$F$ which contains a matching~$M$ saturating at least $\ell=(1/2-5\varepsilon^{1/7})s$ vertices of $R$. Let $W=(i_1, i_2, \dots i_k)$ be a minimal walk contained in $F$ which contains~$M$. Clearly, $F$ is a tree and so $k \le 2(s-1)$. For each $e\in M$ we find the first appearance of $e$ in $W$, say $(i_{j},i_{j+1})$, and replace it by a red path $P_j$ of length $(1-4\varepsilon)n/s$ which alternates between $V_{i_j}$ and $V_{i_{j+1}}$ in $G$. Clearly, 
\[
\sum_{j=1}^{\ell} |P_j| \ge \ell\cdot (1-4\varepsilon)n/s = (1/2-5\varepsilon^{1/7})(1-4\varepsilon)n.
\]
Finally using elementary properties of $(\varepsilon,p)$-regular pairs we glue all $P_j$'s (following the order in $W$) as in the previous case loosing only $poly(\varepsilon)n$ vertices.
\end{proof}

It immediately follows from the above proof that a better constant in Corollary~\ref{cor:erdos_gallai} yields a bigger value of $c_r$.

\section{Large monochromatic components in $\G(n,p)$}\label{sec5}

It is easy to see that in every 2-colouring of the edges of $K_n$ there is a monochromatic connected subgraph on $n$ vertices. For three colours the analogue problem was first solved by Gerencs{\'e}r and Gy{\'a}rf{\'a}s~\cite{GG67} (see also \cite{A70,BB85}). The generalization of this result to any number of colours was proved by Gy{\'a}rf{\'a}s~\cite{G77} and it also follows from a more general result of F\"uredi~\cite{F81}. 

\begin{theorem}[\cite{G77,F81}]\label{thm:G77}
Let $r\in \N \setminus\{1\}$. Suppose that the edges of $K_n$ are coloured with $r$ colours. Then, there is a monochromatic component with at least $n/(r-1)$ vertices. This result is sharp if $r-1$ is a prime power and $(r-1)^2$ divides $n$.
\end{theorem}

In this section we consider a similar problem for $\G(n,p)$. The following was proven by Sp{\"o}hel, Steger and Thomas~\cite{SST10} and also independently by Bohman, Frieze, Krivelevich, Loh and Sudakov~\cite{BFKLS11}. 

\begin{theorem}[\cite{SST10,BFKLS11}]
Let $r\in \N \setminus\{1\}$ and let $\tau_r$ denote the constant which determines the threshold for $r$-orientability of the random graph $\G(n,rc/n)$. Then, for any constant $c > 0$ the following holds a.a.s.
\begin{enumerate}[(i)]
\item If $c < \tau_r$, then there exists an $r$-colouring of the edges of $\G(n,rc/n)$ in which all monochromatic
components have $o(n)$ vertices.
\item If $c > \tau_r$, then every $r$-colouring of the edges of $\G(n,rc/n)$ contains a monochromatic component with $\Theta(n)$ vertices.
\end{enumerate}
\end{theorem}

Here we complement this result considering the case when the average degree tends to infinity (as $n \to \infty$). This time, we are interested in the following constant:
\begin{eqnarray*}
d_r = \sup\{ d \in [0,1] &:& \G(n,p) \ \text{ has a monochromatic component} \\ 
&& \text{ on at least $dn$ vertices a.a.s., provided $np \to \infty$} \}.
\end{eqnarray*}
Clearly $d_r \ge c_r$, where $c_r$ is defined as in the previous section (cf.~\eqref{eq:cr}). 

\begin{theorem}\label{thm:mono_comps}
Let $r\in \N \setminus\{1\}$, $\alpha > 0$ be an arbitrarily small constant, and $p=p(n)$ be such that $pn\to\infty$. Then, a.a.s.\ for any $r$-colouring of the edges of $\G(n,p)$ there is a monochromatic component on at least $(1/(r-1)-\alpha)n$ vertices, which implies that $d_r = 1/(r-1)$. This constant is optimal for infinitely many $r$.
\end{theorem}

First we derive an analogous result to Corollary~\ref{cor:erdos_gallai}, which is a perturbed version of Theorem~\ref{thm:G77}.

\begin{lemma}\label{lem:mono_comps}
Let $r\in \N \setminus\{1\}$ and $0<\varepsilon \le 1/r^2$. Let $G$ be a graph of order $n$ with at least $(1-\varepsilon)\binom{n}{2}$ edges. Then, for any $r$-colouring of the edges of $G$ there is a monochromatic component on at least $(1/(r-1)-\varepsilon r^2)n$ vertices.
\end{lemma}

Let us note that a special case of this result for $r=3$ was obtained by Figaj and \L{}uczak~\cite{FL07}. Our proof is different; we will use the following result of Liu, Morris and Prince~\cite{LMP09}.

\begin{lemma}[Lemma 9 in \cite{LMP09}]\label{claim:bipart}
Let $H=(V_1,V_2,E)$ be a bipartite graph. Assume that $|E| \ge \eta |V_1||V_2|$ for some $\eta>0$. Then, $H$ has a component on at least $\eta(|V_1|+|V_2|)$ vertices. 
\end{lemma}

\begin{proof}[Proof of Lemma~\ref{lem:mono_comps}]
Let $G=(V,E)$ be a graph of order $n$ with at least $(1-\varepsilon)\binom{n}{2}\ge \binom{n}{2} - (\varepsilon/2) n^2$ edges. For a contradiction, suppose that there is a colouring of the edges of $G$ with $r$ colours so that $C$, a largest monochromatic component in $G$, satisfies $|V(C)| < (1/(r-1)-\varepsilon r^2)n$. On the other hand, by Corollary~\ref{cor:erdos_gallai}, $|V(C)|\ge (1/r-\varepsilon)n$.

Consider the bipartite graph $F$ induced by the edges of $G$ between $V(C)$ and $V(G)\setminus V(C)$. Clearly, the edges of $F$ are coloured with at most $r-1$ colours (as the colour of $C$ is not used). First observe that 
\begin{align*}
(\varepsilon/2) n^2 &= |V(C)||V(G)\setminus V(C)| \cdot \frac{\varepsilon n^2}{2|V(C)||V(G)\setminus V(C)|}\\
&\le |V(C)||V(G)\setminus V(C)| \cdot \frac{\varepsilon n^2}{2(1/r-\varepsilon)n \cdot (1-(1/r-\varepsilon))n}.
\end{align*}
Since $\varepsilon \le 1/r^2$ and $r\ge 2$, we get $1/r - \varepsilon = (1-\varepsilon r)/r \ge (1-1/r)/r\ge 1/(2r)$. Thus,
\[
(\varepsilon/2) n^2 \le |V(C)||V(G)\setminus V(C)| \cdot \frac{\varepsilon}{2\cdot 1/(2r) \cdot (r-1)/r}
\le |V(C)||V(G)\setminus V(C)| \cdot \varepsilon r^2.
\]
Consequently,
\[
|E(F)|  \ge |V(C)||V(G)\setminus V(C)| - (\varepsilon/2) n^2
\ge (1 -  \varepsilon r^2)|V(C)||V(G)\setminus V(C)|.
\]
Let $H$ be a subgraph of $F$ induced by the majority colour. Thus,
\[
|E(H)| \ge \frac{1}{r-1} (1 -  \varepsilon r^2)|V(C)||V(G)\setminus V(C)|,
\]
and so Lemma~\ref{claim:bipart} implies that there is a monochromatic component of order 
\[
\frac{1}{r-1} (1 -  \varepsilon r^2)n \ge \left(\frac{1}{r-1} - \varepsilon r^2\right)n,
\]
that is larger than $C$, a largest monochromatic component in $G$. We get the desired contradiction and the proof is finished.
\end{proof}

Finally, we are ready to sketch the proof of the main result of this section.

\begin{proof}[Sketch of the proof of Theorem~\ref{thm:mono_comps}]
This is basically the proof of Theorem~\ref{thm:multi_Gnp} with Corollary~\ref{cor:erdos_gallai} replaced by Lemma~\ref{lem:mono_comps}. 
We find a monochromatic spanning tree on $(1/(r-1) - \varepsilon r^2)s$ vertices in the cluster graph, and then we replace each edge by a long path (in a bipartite graph). All those paths intersect yielding a large monochromatic component.
The sharpness follows immediately from the sharpness of Theorem~\ref{thm:G77}.
\end{proof}

\section{Concluding remarks}

We finish the paper with a few remarks and possible questions for future work. In this paper, we improved both a lower and an upper bound for $\ram{P_n}$, but clearly there is still a lot of work that is waiting to be done. Closing the gap is a natural question. However, it seems that in order to obtain a substantial improvement, one needs to develop a new approach to attack this question. For more colours, as we already mentioned, it is interesting to determine the order of magnitude of $\ram{P_n,r}$ as a function of $r$. Is it exponential in~$r$? Or maybe it is only polynomial in~$r$?

In this paper, we are also concerned with monochromatic paths and components in $\G(n,p)$, provided that $pn \to \infty$. Exactly the same question can be asked for $\G_{n,d}$. It is known, due to a result of Kim and Vu~\cite{KV04}, that if $d \gg \log n$ and $d\ll n^{1/3} / \log^2 n$, then there exists a coupling of $\G(n,p)$ with $p = \frac dn (1-(\log n/d)^{1/3})$, and $\G_{n,d}$, such that a.a.s.\ $\G(n,p)$ is a subgraph of $\G_{n,d}$.  A recent result of Dudek, Frieze, Ruci\'nski, and \v{S}ileikis~\cite{DFRS15} (see also Section 10.3 in~\cite{FK15}) extends that for denser graphs. Consequently, our results for $\G(n,p)$ model imply immediately the counterpart results for $\G_{n,d}$, provided $d \gg \log n$. It would be interesting to investigate the behaviour for $\Omega(1)=d=O(\log n)$.

Finally, determining the value of $c_r$ might be of some interest (cf.~\eqref{eq:cr}). Letzter~\cite{L15} showed that $c_2 = 2/3$ and in this paper we showed that $c_3 = 1/2$. For $r \in \N \setminus \{1,2,3\}$ we proved that $1/r \le c_r \le 1/(r-1)$ but the exact value of $c_r$ still  remains unknown.


\providecommand{\bysame}{\leavevmode\hbox to3em{\hrulefill}\thinspace}
\providecommand{\MR}{\relax\ifhmode\unskip\space\fi MR }
\providecommand{\MRhref}[2]{%
  \href{http://www.ams.org/mathscinet-getitem?mr=#1}{#2}
}
\providecommand{\href}[2]{#2}

\end{document}